\documentclass[11pt]{article}

\usepackage{amssymb,amsmath,amsfonts,amsthm, hyperref, listings}
\usepackage{latexsym}
\usepackage{graphics}
\usepackage{indentfirst}

\allowdisplaybreaks

\newtheorem{innercustomthm}{Theorem}
\newenvironment{customthm}[1]
  {\renewcommand\theinnercustomthm{#1}\innercustomthm}
  {\endinnercustomthm}

\newtheorem*{thm*}{Theorem}
\newtheorem{thm}{Theorem}
\newtheorem{lem}[thm]{Lemma}

\newtheorem{ques}[thm]{Question}

\newcommand{\N}{\mathbb{N}}

\newcommand\numberthis{\addtocounter{equation}{1}\tag{\theequation}}

\begin{document}

\title{On Polynomial Representations of Dual DP Color Functions}

\author{Jeffrey A. Mudrock$^1$ and Gabriel Sharbel$^2$}

\footnotetext[1]{Department of Mathematics and Statistics, University of South Alabama, Mobile, AL 36688.  E-mail:  {\tt {mudrock@southalabama.edu}}}
\footnotetext[2]{School of Computing, University of South Alabama, Mobile, AL 36688.  E-mail:  {\tt {gss2121@jagmail.southalabama.edu}}}

\maketitle

\begin{abstract}

DP-coloring (also called correspondence coloring) is a generalization of list coloring that was introduced by Dvo\v{r}\'{a}k and Postle in 2015. The chromatic polynomial of a graph is an important notion in algebraic combinatorics that was introduced by Birkhoff in 1912; denoted $P(G,m)$, it equals the number of proper $m$-colorings of graph $G$. Counting function analogues of chromatic polynomials have been introduced for list colorings:  $P_{\ell}$, list color functions (1990); DP colorings: $P_{DP}$, DP color functions (2019), and $P^*_{DP}$, dual DP color functions (2021). For any graph $G$ and $m \in \N$, $P_{DP}(G, m) \leq P_\ell(G,m) \leq P(G,m) \leq P_{DP}^*(G,m)$.  In 2022 (improving on older results) Dong and Zhang showed that for any graph $G$, $P_{\ell}(G,m)=P(G,m)$ whenever $m \geq |E(G)|-1$.  Consequently, the list color function of a graph is a polynomial for sufficiently large $m$.  One of the most important and longstanding open questions on DP color functions asks: for every graph $G$ is there an $N \in \N$ and a polynomial $p(m)$ such that $P_{DP}(G,m) = p(m)$ whenever $m \geq N$?  We show that the answer to the analogue of this question for dual DP color functions is no.  Our proof reveals a connection between a dual DP color function and the balanced chromatic polynomial of a signed graph introduced by Zaslavsky in 1982.  

\medskip



\end{abstract}

\section{Introduction}\label{intro}

In this paper all graphs are nonempty, finite, simple graphs unless otherwise noted.  Generally speaking we follow West~\cite{W01} for terminology and notation.  The set of natural numbers is $\N = \{1,2,3, \ldots \}$.  For $m \in \N$, we write $[m]$ for the set $\{1, \ldots, m \}$.  If $G$ is a graph and $S, U \subseteq V(G)$, we use $G[S]$ for the subgraph of $G$ induced by $S$, and we use $E_G(S, U)$ for the set consisting of all the edges in $E(G)$ that have one endpoint in $S$ and the other in $U$.  If an edge in $E(G)$ connects the vertices $u$ and $v$, the edge can be represented by $uv$ or $vu$.  

\subsection{List Coloring and DP-Coloring} \label{basic}

In the classical vertex coloring problem we wish to color the vertices of a graph $G$ with up to $m$ colors from $[m]$ so that adjacent vertices receive different colors, a so-called \emph{proper $m$-coloring}. The chromatic number of a graph $G$, denoted $\chi(G)$, is the smallest $m$ such that $G$ has a proper $m$-coloring.  List coloring was introduced independently by Vizing~\cite{V76} and Erd\H{o}s, Rubin, and Taylor~\cite{ET79} in the 1970s.  For list coloring, we associate a \emph{list assignment} $L$ with a graph $G$ which assigns to each $v \in V(G)$ a list of colors $L(v)$.  Then, $G$ is \emph{$L$-colorable} if there exists a proper coloring $f$ of $G$ such that $f(v) \in L(v)$ for each $v \in V(G)$ (we refer to $f$ as a \emph{proper $L$-coloring} of $G$).  A list assignment $L$ is called a \emph{$k$-assignment} for $G$ if $|L(v)|=k$ for each $v \in V(G)$.  The \emph{list chromatic number} of a graph $G$, denoted $\chi_\ell(G)$, is the smallest $k$ such that $G$ is $L$-colorable whenever $L$ is a $k$-assignment for $G$.  Since $G$ must be $L$-colorable when $L$ is a $\chi_\ell(G)$-assignment for $G$ that assigns the same list of colors to each element in $V(G)$, it is clear that $\chi(G) \leq \chi_\ell(G)$.  

In 2015, Dvo\v{r}\'{a}k and Postle~\cite{DP15} introduced a generalization of list coloring called DP-coloring (they called it correspondence coloring) in order to prove that every planar graph without cycles of lengths 4 to 8 has list chromatic number at most 3. DP-coloring has been extensively studied over the past 9 years (see e.g.,~\cite{BH21, B16, BK182, BK17, BK18, KM19, KM20, KO18, LLYY19, M18}). Intuitively, DP-coloring is a variation on list coloring where each vertex in the graph still gets a list of colors, but identification of which colors are the same can vary from edge to edge.  Formally, for a graph $G$, a \emph{DP-cover} (or simply a \emph{cover}) of $G$ is an ordered pair $\mathcal{H}=(L,H)$, where $H$ is a graph and $L:V(G)\to 2^{V(H)}$ is a function satisfying the following conditions: 
    \begin{itemize}
		\item $\{L(v) : v \in V(G)\}$ is a partition of $V(H)$ into $|V(G)|$ parts, 

		\item for every pair of adjacent vertices $u$, $v\in V(G)$, the edges in $E_H\left(L(u),L(v)\right)$ form a matching (possibly empty), and

		\item $\displaystyle E(H) = \bigcup_{uv \in E(G)} E_{H}(L(u),L(v)).$
    \end{itemize}
    
Suppose $\mathcal{H}=(L,H)$ is a cover of a graph $G$.  A \emph{transversal} of $\mathcal{H}$ is a set of vertices $T\subseteq V(H)$ containing exactly one vertex from $L(v)$ for each $v \in V(G)$. A transversal $T$ is said to be \emph{independent} if $T$ is an independent set in $H$.  If $\mathcal{H}$ has an independent transversal $T$, then $T$ is said to be a \emph{proper $\mathcal{H}$-coloring} of $G$, and $G$ is said to be \emph{$\mathcal{H}$-colorable}.  An \emph{$m$-fold cover} of $G$ is a cover $\mathcal{H}=(L,H)$ such that $|L(v)|=m$ for all $v\in V(G)$.  An $m$-fold cover $\mathcal{H} = (L,H)$ of a graph $G$ is called \emph{full} if for each $uv \in E(G)$, $|E_{H}(L(u),L(v))| = m$.  The \emph{DP-chromatic number} of a graph $G$, denoted $\chi_{DP}(G)$, is the smallest $m$ such that $G$ is $\mathcal{H}$-colorable whenever $\mathcal{H}$ is an $m$-fold cover of $G$.

Given a cover $\mathcal{H} = (L,H)$ of $G$ and subgraph $G'$ of $G$, the \emph{subcover of $\mathcal{H}$ corresponding to $G'$} is $\mathcal{H}'=(L',H')$ where $L'$ is the restriction of $L$ to $V(G')$ and $H'$ is the subgraph of $H$ with vertex set $V(H')=\bigcup_{u\in V(G')}L(u)$ that retains those and only those edges of $H$ that belong to the matchings corresponding to the edges of $G'$.

Suppose $\mathcal{H} = (L,H)$ is an $m$-fold cover of $G$.  We say that $\mathcal{H}$ has a \emph{canonical labeling} if it is possible to name the vertices of $H$ so that $L(u) = \{ (u,j) : j \in [m] \}$ and $(u,j)(v,j) \in E(H)$ for each $j \in [m]$ whenever $uv \in E(G)$.  Clearly, when $\mathcal{H}$ has a canonical labeling, $G$ has an $\mathcal{H}$-coloring if and only if $G$ has a proper $m$-coloring.  Also, given an $m$-assignment, $L$, for a graph $G$, it is easy to construct an $m$-fold cover $\mathcal{H}'$ of $G$ such that $G$ has an $\mathcal{H}'$-coloring if and only if $G$ has a proper $L$-coloring (see~\cite{BK17}).  It follows that $\chi(G) \leq \chi_\ell(G) \leq \chi_{DP}(G)$.  
 
\subsection{Counting Proper Colorings, List Colorings, and DP-Colorings}

In 1912 Birkhoff introduced the notion of the chromatic polynomial of a graph in hopes of using it to make progress on the four color problem.  For $m \in \N$, the \emph{chromatic polynomial} of a graph $G$, $P(G,m)$, is the number of proper $m$-colorings of $G$.  It can be shown that $P(G,m)$ is a polynomial in $m$ of degree $|V(G)|$ (see~\cite{B12}).  For example, for any $n \in \N$, $P(K_n,m) = \prod_{i=0}^{n-1} (m-i)$.

The notion of chromatic polynomial was extended to list coloring in the 1990s~\cite{AS90}.   In particular, if $L$ is a list assignment for $G$, we use $P(G,L)$ to denote the number of proper $L$-colorings of $G$. The \emph{list color function} of $G$, denoted $P_\ell(G,m)$, is the minimum value of $P(G,L)$ where the minimum is taken over all possible $m$-assignments $L$ for $G$.  It is clear that $P_\ell(G,m) \leq P(G,m)$ for each $m \in \N$ since we must consider an $m$-assignment that assigns the same $m$ colors to all the vertices in $G$ when considering all possible $m$-assignments for $G$.  In general, the list color function can differ significantly from the chromatic polynomial for small values of $m$.  However, for large values of $m$, Dong and Zang~\cite{DZ22} (improving upon results in~\cite{D92}, \cite{T09}, and~\cite{WQ17}) showed the following.

\begin{thm} [\cite{DZ22}] \label{thm: DZ22}
For any graph $G$, $P_{\ell}(G,m)=P(G,m)$ whenever $m \geq |E(G)|-1$.
\end{thm}

In 2019, Kaul and the first author introduced a DP-coloring analogue of the chromatic polynomial of a graph in hopes of using it as a tool for making progress on some open questions related to the list color function~\cite{KM19}.  Since its introduction in 2019, DP color functions have received some attention in the literature (see e.g.,~\cite{BH21, DKM22, DY21, HK21, KM21, LY22, MT20, PS23}).  Suppose $\mathcal{H} = (L,H)$ is a cover of graph $G$.  Let $P_{DP}(G, \mathcal{H})$ be the number of $\mathcal{H}$-colorings of $G$.  Then, the \emph{DP color function} of $G$ $P_{DP}(G,m)$ is the minimum value of $P_{DP}(G, \mathcal{H})$ where the minimum is taken over all possible $m$-fold covers $\mathcal{H}$ of $G$.~\footnote{We take $\N$ to be the domain of the DP color function of any graph.}  It is easy to show that for any graph $G$ and $m \in \N$, $P_{DP}(G, m) \leq P_\ell(G,m) \leq P(G,m)$.

Interestingly, unlike the list color function, it is known that $P_{DP}(G,m)$ does not necessarily equal $P(G,m)$ for sufficiently large $m$.  Indeed, in~\cite{KM19} it is shown that if $G$ is a graph with girth that is even, then there is an $N \in \N$ such that $P_{DP}(G, m) < P(G,m)$ whenever $m \geq N$ (this result was further generalized by Dong and Yang in~\cite{DY21}).  This leads to a longstanding open question about DP color functions that served as one of the motivations for this paper. 

\begin{ques} [\cite{KM19}] \label{ques: polynomial}
For any graph $G$ does there always exist an $N \in \N$ and a polynomial $p(m)$ such that $P_{DP}(G,m) = p(m)$ whenever $m \geq N$?
\end{ques}

Recently, the first author~\cite{M21} introduced the \emph{dual DP color function} of a graph $G$, denoted $P^{*}_{DP}(G,m)$, which equals the maximum value of $P_{DP}(G, \mathcal{H})$ where the maximum is taken over all full $m$-fold covers $\mathcal{H}$ of $G$.  Every graph $G$ has a full $m$-fold cover with a canonical labeling. So, $P_{DP}(G,m) \leq P_{\ell}(G,m) \leq P(G,m) \leq P^{*}_{DP}(G,m)$.  In addition to providing an upper bound on the chromatic polynomial, the dual DP color function of a graph also appears in a lower bound for the DP color function of the graph based on a deletion-contraction relation (see~\cite{M21}).

In this paper, we show that the answer to the analogue of Question~\ref{ques: polynomial} for the dual DP color function is no.  We hope that our approach can provide new insights into answering Question~\ref{ques: polynomial}.

\subsection{Summary of Results}

The results in this paper came from studying dual DP color functions of complete graphs.  It is known that $P_{DP}^*(K_2,m) = m(m-1)$ for each $m \in \N$ and $P_{DP}^*(K_3,m) = (m-1)^3 + 1$ for each $m \geq 2$ (see~\cite{M21}).  In Section~\ref{main}, we prove the following.

\begin{thm} \label{thm: main}
Suppose $m \geq 2$.  Then,
\[P^*_{DP}(K_4,m) = \begin{cases} 
      m^4 - 6m^3 + 15m^2 - 13m & \text{if $m$ is even} \\
      m^4 - 6m^3 + 15m^2 - 13m-3 & \text{if $m$ is odd.} 
   \end{cases}
\]  
Consequently, there is no $N \in \N$ and polynomial $p(m)$ such that $P_{DP}^*(K_4,m) = p(m)$ whenever $m \geq N$.
\end{thm}

Clearly, Theorem~\ref{thm: main} shows that the answer to the analogue of Question~\ref{ques: polynomial} for the dual DP color function is no.  Interestingly, Theorem~\ref{thm: main} also reveals a connection between the dual DP color function of $K_4$ and the \emph{balanced chromatic polynomial} of a signed version of $K_4$ (balanced chromatic polynomials were introduced by Zaslavsky in 1982~\cite{Z82}).  

Using the notation in~\cite{CL22}, we now make this connection explicit.  Suppose $G$ is a graph.  A \emph{signed graph} on $G$ is a pair $SG = (G, \sigma)$ where the \emph{sign} on each edge of $SG$ is given by the function $\sigma : E(G) \rightarrow \{1, -1\}$.  For $\lambda \geq 2$, a \emph{$\lambda$-coloring} of $SG$ is a map $f: V(G) \rightarrow K$ where $K = \{-t, \ldots, 1,0,1, \ldots, t\}$ when $\lambda = 2t+1$ for some $t \in \N$ and $K = \{-t, \ldots, 1,1, \ldots, t\}$ when $\lambda = 2t$ for some $t \in \N$.  Additionally, $f$ is a \emph{proper $\lambda$-coloring} of $SG$ if for each $uv \in E(G)$, $f(u) \neq \sigma(uv)f(v)$.  For each $\lambda \geq 2$, let $P_{SG}(\lambda)$ denote the number of proper $\lambda$-colorings of $SG$.  The \emph{balanced chromatic polynomial} of $SG$ is then $P_{SG}(2t)$ for each $t \in \N$ (it can be proven $P_{SG}(2t)$ is a polynomial in $t$).  Finally, suppose $G = K_4$, $\sigma: E(G) \rightarrow \{1, -1\}$ is identically -1, and $SG = (G, \sigma)$.  Then, it can be verified that Theorem~\ref{thm: main} implies $P^*_{DP}(K_4,2l) = P_{SG}(2l)$ for each $l \in \N$.

The following two questions are now natural.

\begin{ques} \label{ques: connection}
Suppose $n \in \N$ and $G = K_n$.  Let $\sigma: E(G) \rightarrow \{1, -1\}$ be identically -1, and let $SG = (G, \sigma)$.  Is it the case that $P^*_{DP}(K_n,2l) = P_{SG}(2l)$ for each $l \in \N$?  
\end{ques}

\begin{ques} \label{ques: general}
For each $n \geq 4$, is it the case that there is no $N \in \N$ and polynomial $p(m)$ such that $P_{DP}^*(K_n,m) = p(m)$ whenever $m \geq N$?
\end{ques}

It is easy to verify the answer to Question~\ref{ques: connection} is yes when $n \in [3]$, and Theorem~\ref{thm: main} implies the answer is yes when $n=4$.  In proving Theorem~\ref{thm: main} we will discover that the interaction between the subcovers corresponding to the three and four cycles contained in $K_4$ is quite important.  So, our final result, which we prove in Section~\ref{more}, can be viewed as progress toward both of the Questions above.

\begin{thm} \label{thm: generalize}
Suppose $G = K_n$ with $n \geq 4$, and let $t = \binom{n}{2}$.  If $m > 2^{t+1} + t + n - 6$ and $\mathcal{H} = (L,H)$ is a full $m$-fold cover of $G$ satisfying $P_{DP}(G, \mathcal{H}) = P_{DP}^*(G,m)$, then $H$ is triangle-free.  Moreover, when $m > 2^{t+1} + t + n - 6$,
$$f(m) - 2^{t}m^{n-4} \leq P^*_{DP}(G,m) \leq  f(m) + 2^{t}m^{n-4}$$
where 
$$f(m) = m^n - t m^{n-1} + \binom{t}{2} m^{n-2} - \left(\binom{t}{3} - \binom{n}{3} - 3\binom{n}{4} \right) m^{n-3}.$$
\end{thm}

Note that if $l, n \in \N$, $n \geq 4$, $G = K_n$, $\sigma: E(G) \rightarrow \{1, -1\}$ is identically -1, and $SG = (G, \sigma)$, then Theorem~\ref{thm: generalize} implies that $P_{DP}^*(G,2l) - P_{SG}(2l) = O(l^{n-4})$ as $l \rightarrow \infty$.

\section{Proof of Theorem~\ref{thm: main}} \label{main}

This section is organized as follows.  We begin by showing that the formula in Theorem~\ref{thm: main} is an upper bound on the dual DP color function of $K_4$.  Then, we construct covers to show that the upper bound is attainable.  The key to establishing the upper bound is generalizing the proof technique of the following classical result to the context of DP-coloring.

\begin{thm} [\cite{B12}] \label{pro: base}
Suppose $G$ is a graph.  Then,
$$P(G,m) = \sum_{A \subseteq E(G)} (-1)^{|A|} m^{k_A}$$
where $k_A$ is the number of components of the spanning subgraph of $G$ with edge set $A$.
\end{thm} 

We will now introduce some notation that will be used for the remainder of this section.  Suppose that $G = K_4$, $V(G) = \{v_1, v_2, v_3, v_4 \}$, and $E(G) = \{e_1, e_2, e_3, e_4, e_5, e_6 \}$.  Also, for some $m \geq 4$ we always have that $\mathcal{H}= (L,H)$ is a full $m$-fold cover of $G$ such that $L(v) = \{(v,i): i \in [m] \}$ for each $v \in V(G)$ (we will handle $m < 4$ computationally in Appendix~\ref{code}).  Additionally, for $i, j \in [4]$ with $i < j$, we let $\sigma_{i,j}$ be the permutation of $[m]$ that maps each $q \in [m]$ to the second coordinate of the vertex in $L(v_j)$ that is adjacent to $(v_i,q)$ in $H$.  

Let $\mathcal{U} = \{ I \subseteq V(H) : |L(v) \cap I| = 1 \text{ for each } v \in V(G) \}$.  Clearly, $|\mathcal{U}| = m^4$.  Now, for $i \in [6]$, if $e_i=v_r v_s$, let $S_i$ be the set consisting of each $I \in \mathcal{U}$ with the property that $H[I]$ contains an edge in $E_H(L(v_r), L(v_s))$.  Also, for each $i \in [6]$ let $C_i = \mathcal{U} - S_i$.  Clearly,
$$P_{DP}(G, \mathcal{H}) = \left | \bigcap_{i=1}^6 C_i \right |.$$
So, by the Inclusion-Exclusion Principle, we see that
$$P_{DP}(G, \mathcal{H}) = |\mathcal{U}| - \left | \bigcup_{i=1}^6 S_i \right | =  m^4 - \sum_{k=1}^6 (-1)^{k-1} \left ( \sum_{1 \leq i_1 < \cdots < i_k \leq 6} \left | \bigcap_{j=1}^k S_{i_j} \right| \right).$$  
We now establish some bounds on the terms in this formula.
\begin{lem} [\cite{MT20}] \label{lem: formulas2}
Assuming the set-up established above, the following two statements hold. \\
(i) For any $k \in [2]$ and $i_1, \ldots, i_k \in [6]$ satisfying $i_1 < \cdots < i_k$, $\left | \bigcap_{j=1}^k S_{i_j} \right|  = m^{4-k}$. \\  
(ii) If $e_{i_1}, \ldots, e_{i_3}$ are distinct edges that are not the edges of a $3$-cycle in $G$, $\left | \bigcap_{j=1}^3 S_{i_j} \right| = m$.  
\end{lem}

\begin{lem} \label{lem: subgraph}
Assuming the set-up established above, suppose $k \in \{3,4,5,6\}$, and $i_1, \ldots, i_k \in [6]$ satisfy $i_1 < \cdots < i_k$.  Let $G'$ be the subgraph of $G$ with edge set $\{e_{i_1}, \ldots, e_{i_k} \}$ and vertex set consisting of only the vertices that are an endpoint of at least one element in $E(G')$.  Suppose $\mathcal{H}' = (L',H')$ is the subcover of $\mathcal{H}$ corresponding to $G'$.  Let $C(G')$ be the number of copies of $G'$ contained in $H'$.  If $k=3$ and $e_{i_1}, \ldots, e_{i_3}$ are the edges of a 3-cycle in $G$, then 
$$\left | \bigcap_{j=1}^3 S_{i_j} \right| = mC(G') \text{; otherwise,  } \left | \bigcap_{j=1}^k S_{i_j} \right| = C(G').$$    
\end{lem}

\begin{proof}
Let $\mathcal{C}(G')$ be the set consisting of the vertex set of each copy of $G'$ in $H'$.  Note that any element of $\mathcal{C}(G')$ is a transversal of $\mathcal{H}'$.  Also, $\bigcap_{j=1}^k S_{i_j}$ has an element if and only if $\mathcal{C}(G')$ has an element. So, both results are clear when $C(G') = 0$, and we suppose that $C(G') > 0$. 

First, suppose $k=3$ and $e_{i_1}, \ldots, e_{i_3}$ are the edges of a 3-cycle in $G$.  Suppose $\{v_q\} = V(G) - V(G')$.  Let $M : \mathcal{C}(G') \times [m] \rightarrow \bigcap_{j=1}^3 S_{i_j}$ be the function given by $M(V,x) = V \cup \{(v_q,x) \}$.  It is easy to verify that $M$ is a bijection which means that $\left | \bigcap_{j=1}^3 S_{i_j} \right| = mC(G')$.

In the case $k \neq 3$, or $k=3$ and $e_{i_1}, \ldots, e_{i_3}$ are not the edges of a 3-cycle in $G$, $\mathcal{C}(G') = \bigcap_{j=1}^k S_{i_j}$ which completes the proof.
\end{proof}

Now, suppose that $T_1, T_2, T_3,$ and $T_4$ are the 3-cycles contained in $G$.  For each $i \in [4]$, let $t_i$ be the number of $3$-cycles contained in the second coordinate of the subcover of $\mathcal{H}$ corresponding to $T_i$.  Note that the vertex set of each such $3$-cycle is a transversal of the subcover of $\mathcal{H}$ corresponding to $T_i$.   

\begin{lem} \label{lem: plusedge}
Suppose $G'$ is the spanning subgraph of $G$ such that $E(G')$ contains all the edges of $T_i$ for some $i \in [4]$ and one additional edge $e_r$.  Suppose $\mathcal{H}' = (L',H')$ is the subcover of $\mathcal{H}$ corresponding to $G'$.  Then, $C(G') = t_i$ where $C(G')$ is the number of copies of $G'$ contained in $H'$
\end{lem}

\begin{proof}
Let $\mathcal{C}(G')$ be the set consisting of the vertex set of each copy of $G'$ in $H'$, and let $\mathcal{T}$ be the set consisting of the vertex set of each copy of $T_i$ contained in the second coordinate of the subcover of $\mathcal{H}$ corresponding to $T_i$.  Since any element of $\mathcal{C}(G')$ is a transversal of $\mathcal{H}'$, it is clear that if $\mathcal{C}(G') = \emptyset$, then $\mathcal{T} = \emptyset$.  So, we may assume $C(G') > 0$.  

Suppose $\{v_q\} = V(G) - V(T_i)$ (note $v_q$ must be an endpoint of $e_r$).  Let $M: \mathcal{C}(G') \rightarrow \mathcal{T}$ be the function that maps each $I \in \mathcal{C}(G')$ to the set obtained from $I$ by deleting the element in $I$ with first coordinate $v_q$.  It is easy to verify that $M$ is a bijection, and the result follows.
\end{proof}

Now, suppose that $M_1, \ldots, M_6$ are the copies of $K_4$ minus an edge contained in $G$.  For each $i \in [6]$, let $m_i$ be the number of copies of $M_i$ contained in the second coordinate of the subcover of $\mathcal{H}$ corresponding to $M_i$.   Note that the vertex set of each such copy is a transversal of the subcover of $\mathcal{H}$ corresponding to $M_i$.

\begin{lem} \label{lem: kite}
Suppose $G' = M_i$ for some $i \in [6]$, and assume $E(G')$ contains all the edges of $T_j$ and $T_r$ for some $j,r \in [4]$ with $j \neq r$.  Then, $m_i \leq \min\{t_j , t_r\}$.
\end{lem}

\begin{proof}
Since the result is clear when $m_i = 0$, we suppose $m_i > 0$.  Suppose $\mathcal{H}' = (L',H')$ is the subcover of $\mathcal{H}$ corresponding to $G'$.  Let $\mathcal{C}(G')$ be the set consisting of the vertex set of each copy of $G'$ in $H'$, and let $\mathcal{T}_j$ (resp. $\mathcal{T}_r$) be the set consisting of the vertex set of each copy of $T_j$ (resp. $T_r$) contained in the second coordinate of the subcover of $\mathcal{H}$ corresponding to $T_j$ (resp. $T_r$). Now, let
$$\mathcal{T} = \{V_1 \cup V_2 : V_1 \in \mathcal{T}_j, V_2 \in \mathcal{T}_r, |V_1 \cap V_2| = 2 \}.$$
It is easy to see that $\mathcal{C}(G') = \mathcal{T}$.  Then since for each $V \in \mathcal{T}_j$ there is at most one $V' \in \mathcal{T}_r$ such that $|V \cap V'|=2$ (this also holds when $j$ and $r$ are interchanged), $\min\{t_j,t_r \} \geq |\mathcal{T}|$.
\end{proof}

Next, suppose $Q_1, \ldots, Q_3$ are the 4-cycles contained in $G$.  For each $i \in [3]$, let $q_i$ be the number of $4$-cycles contained in the second coordinate of the subcover of $\mathcal{H}$ corresponding to $Q_i$.  Note that the vertex set of each such $4$-cycle is a transversal of the subcover of $\mathcal{H}$ corresponding to $Q_i$.  Finally let $z$ be the number of copies of $K_4$ contained in $H$.  Using this notation along with Lemmas~\ref{lem: formulas2}, \ref{lem: subgraph}, and~\ref{lem: plusedge} we see:

\begin{align*}
&P_{DP}(G, \mathcal{H}) \\
&=  m^4 + \sum_{k=1}^6 (-1)^{k} \left ( \sum_{1 \leq i_1 < \cdots < i_k \leq 6} \left | \bigcap_{j=1}^k S_{i_j} \right| \right) \\
&= m^4 - 6m^3 + 15m^2 - 16m - m(t_1+t_2+t_3+t_4) + \sum_{k=4}^6 (-1)^{k} \left ( \sum_{1 \leq i_1 < \cdots < i_k \leq 6} \left | \bigcap_{j=1}^k S_{i_j} \right| \right) \\
&= m^4 - 6m^3 + 15m^2 - 16m + (3-m)\sum_{i=1}^4 t_i + \sum_{i=1}^3q_i - \sum_{i=1}^6 m_i + z. \numberthis \label{identity}
\end{align*}

Using Lemma~\ref{lem: kite}, and noting that $q_i \leq m$ for each $i \in [3]$ and $\min \{t_1,t_2, t_3, t_4 \} \geq z$, we obtain the following.

\begin{lem} \label{lem: upper}
For each $m \geq 4$,
$$P^*_{DP}(K_4,m) \leq m^4 - 6m^3 + 15m^2 - 13m.$$ 
Moreover, using the notation above, if $\mathcal{H}$ is a full $m$-fold cover of $G = K_4$ such that $t_i = 0$ for each $i \in [4]$ and $q_j = m$ for each $j \in [3]$, then
$$P_{DP}(G, \mathcal{H}) = m^4 - 6m^3 + 15m^2 - 13m.$$
\end{lem}

We are now ready to prove two important results that apply when $m$ is odd.  These results will allow us to improve on the upper bound in Lemma~\ref{lem: upper} in the case that $m$ is odd.

\begin{lem} \label{lem: odd}
Suppose $m \geq 5$ and $m$ is odd.  If $\mathcal{H}$ is a full $m$-fold cover of $G = K_4$ such that $q_j = m$ for each $j \in [3]$, then $\sum_{i=1}^4 t_i > 0$.
\end{lem}

\begin{proof}
For the sake of contradiction suppose $\mathcal{H} = (L,H)$ is a full $m$-fold cover of $G = K_4$ such that $q_j = m$ for each $j \in [3]$ and $\sum_{i=1}^4 t_i = 0$ (equivalently $t_i = 0$ for each $i \in [4]$).  Using the notation above, suppose: the vertices of $Q_1$ in cyclic order are $v_1, v_2, v_3, v_4$, the vertices of $Q_2$ in cyclic order are $v_1, v_2, v_4, v_3$, and the vertices of $Q_3$ in cyclic order are $v_1, v_3, v_2, v_4$.

Since $q_1=m$ we may assume that the vertices of $H$ are named so that $\sigma_{1,2}$, $\sigma_{2,3}$, $\sigma_{3,4}$, and $\sigma_{1,4}$ are all the identity permutation.  Since $\sum_{i=1}^4 t_i = 0$, we know that both $\sigma_{1,3}$ and $\sigma_{2,4}$ have no fixed points.  Now, suppose that $x$ is an arbitrary element of $[m]$.  Also, suppose that $\sigma_{1,3}(x) = y$.  Since $q_2 = m$, $\sigma_{1,3}(x) = y$ implies $\sigma^{-1}_{2,4}(y) = x$.  Also, since $q_3 = m$, $\sigma_{2,4}(y) = x$.  So, $\sigma_{2,4}(\sigma_{2,4}(x)) = x$.  Since $x$ was arbitrary, we have that $\sigma_{2,4}$ is an involution.  Since $\sigma_{2,4}$ is a permutation of $[m]$ with $m$ odd that is also an involution, $\sigma_{2,4}$ must have a fixed point which is a contradiction.  
\end{proof}

\begin{lem} \label{lem: oddagain}
Suppose $m \geq 5$ and $m$ is odd.  If $\mathcal{H}$ is a full $m$-fold cover of $G = K_4$ such that $\sum_{i=1}^4 t_i = 0$, then $\sum_{j=1}^3 q_j \leq 3m-3$.
\end{lem}

\begin{proof}
For the sake of contradiction suppose $\mathcal{H} = (L,H)$ is a full $m$-fold cover of $G = K_4$ such that $\sum_{i=1}^4 t_i = 0$ and $\sum_{j=1}^3 q_j \geq 3m-2$.  Note that Lemma~\ref{lem: odd} implies that $\sum_{j=1}^3 q_j \in \{3m-1, 3m-2\}$.  Using the notation above, suppose: the vertices of $Q_1$ in cyclic order are $v_1, v_2, v_3, v_4$, the vertices of $Q_2$ in cyclic order are $v_1, v_2, v_4, v_3$, and the vertices of $Q_3$ in cyclic order are $v_1, v_3, v_2, v_4$.

First, we claim that $q_j \neq m-1$ for each $j \in [3]$.  To see why suppose that $q_1 = m-1$.  Then, we may assume without loss of generality that $\sigma_{1,2}(i) = i$, $\sigma_{2,3}(i) = i$, $\sigma_{3,4}(i) = i$, and $\sigma^{-1}_{1,4}(i) = i$ for each $i \in [m-1]$.  Since $\mathcal{H}$ is full, this implies that $\sigma_{1,2}$, $\sigma_{2,3}$, $\sigma_{3,4}$, and $\sigma_{1,4}$ are all the identity permutation on $[m]$ which implies $q_1 =m$ contradicting $q_1 = m-1$.

Since $q_j \neq m-1$ for each $j \in [3]$, we may suppose that $\sum_{j=1}^3 q_j = 3m-2$, $q_1 = m$, $q_2=m$, and $q_3 = m-2$.  As in the proof of Lemma~\ref{lem: odd} assume that the vertices of $H$ are named so that $\sigma_{1,2}$, $\sigma_{2,3}$, $\sigma_{3,4}$, and $\sigma_{1,4}$ are all the identity permutation on $[m]$.  Since $\sum_{i=1}^4 t_i = 0$, we know that both $\sigma_{1,3}$ and $\sigma_{2,4}$ have no fixed points.  As in the proof of Lemma~\ref{lem: odd}, $q_2=m$ implies that $\sigma_{1,3}(x) = \sigma_{2,4}(x)$ for each $x \in [m]$.  Since $q_3 = m-2$, there are $b_1, b_2 \in [m]$ with $b_1 \neq b_2$ such that $\sigma_{2,4}(\sigma_{1,3}(b_i)) \neq b_i$ for each $i \in [2]$, and $\sigma_{2,4}(\sigma_{1,3}(t)) = t$ for each $t \in ([m] - \{b_1,b_2\})$.  So, we have that $\sigma_{2,4}(\sigma_{2,4}(b_i)) \neq b_i$ for each $i \in [2]$, $\sigma_{2,4}(\sigma_{2,4}(t)) = t$ for each $t \in ([m] - \{b_1,b_2\})$, and $\sigma_{2,4}(x) \neq x$ for each $x \in [m]$.

Since $\sigma_{2,4}(\sigma_{2,4}(b_i)) \neq b_i$ for each $i \in [2]$ we know that we can't have: $\sigma_{2,4}(b_1) = b_2$ and $\sigma_{2,4}(b_2) = b_1$.  So, we may assume without loss of generality $\sigma_{2,4}(b_1) = s$ for some $s \in ([m] - \{b_1,b_2\})$.  We then have that $\sigma_{2,4}(s) \neq s$ and $\sigma_{2,4}(s) \neq b_1$.  So suppose $\sigma_{2,4}(s) = q$ for some $q \in ([m] - \{s, b_1\})$.  This however implies $\sigma_{2,4}(q) = \sigma_{2,4}(\sigma_{2,4}(s)) = s = \sigma_{2,4}(b_1)$ which contradicts the fact that $\sigma_{2,4}$ is a permutation.      
\end{proof}

\begin{lem} \label{lem: upperodd}
For each odd $m \geq 7$,
$$P^*_{DP}(K_4,m) \leq m^4 - 6m^3 + 15m^2 - 13m - 3.$$ 
\end{lem}

\begin{proof}
Suppose $\mathcal{H} = (L,H)$ is a full $m$-fold cover of $G=K_4$ such that $P_{DP}(G, \mathcal{H}) = P_{DP}^*(G,m)$.   Using the notation established above, we have from Identity~(\ref{identity}):
$$ P_{DP}^*(G,m) = m^4 - 6m^3 + 15m^2 - 16m + (3-m)\sum_{i=1}^4 t_i + \sum_{i=1}^3q_i - \sum_{i=1}^6 m_i + z.$$
Notice that since $m \geq 7$, if $\sum_{i=1}^4 t_i > 0$,  $(3-m)\sum_{i=1}^4 t_i - \sum_{i=1}^6 m_i + z \leq (3-m)\sum_{i=1}^4 t_i + \min\{t_1,t_2,t_3,t_4 \} \leq -3t_1 - 4t_2 - 4t_3 - 4t_4 \leq -3.$  On the other hand, if $\sum_{i=1}^4 t_i = 0$, Lemma~\ref{lem: kite} implies $(3-m)\sum_{i=1}^4 t_i - \sum_{i=1}^6 m_i + z = 0$. 

With these two observations in mind, we notice that if $\sum_{i=1}^3 t_i = 0$, Lemma~\ref{lem: oddagain} implies $\sum_{i=1}^3 q_i \leq 3m-3$ which establishes the desired bound.  We also notice that if $\sum_{i=1}^4 t_i > 0$, $P^*_{DP}(G,\mathcal{H}) \leq  m^4 - 6m^3 + 15m^2 - 16m -3 + \sum_{i=1}^3q_i \leq  m^4 - 6m^3 + 15m^2 - 13m - 3$.
\end{proof}

We can now prove Theorem~\ref{thm: main} by proving that the bounds in Lemmas~\ref{lem: upper} and~\ref{lem: upperodd} are attainable.

\begin{customthm} {\ref{thm: main}}
Suppose $m \geq 2$.  Then,
\[P^*_{DP}(K_4,m) = \begin{cases} 
      m^4 - 6m^3 + 15m^2 - 13m & \text{if $m$ is even} \\
      m^4 - 6m^3 + 15m^2 - 13m-3 & \text{if $m$ is odd.} 
   \end{cases}
\]  
Consequently, there is no $N \in \N$ and polynomial $p(m)$ such that $P_{DP}^*(K_4,m) = p(m)$ whenever $m \geq N$.
\end{customthm}

\begin{proof}
Note that the result can be computationally verified when $m=2,3,4,5$ (see Appendix~\ref{code}).  So, suppose that $m$ is even and $m \geq 6$.  We will now construct a full $m$-fold cover, $\mathcal{H} = (L,H)$, of $G = K_4$.  For each $i \in [4]$, let $L(v_i) = \{(v_i,j) : j \in [m] \}$ and let $V(H) = \bigcup_{i=1}^4 L(v_i)$.  Finally, create edges in $H$ so that whenever $i, j \in [4]$ and $i < j$, $(v_i,t)(v_j,t+1) \in E(H)$ when $t$ is odd and $(v_i,t)(v_j,t-1) \in E(H)$ when $t$ is even.  Now, using the notation established above, it is easy to check that $\mathcal{H}$ is a full $m$-fold cover of $G$ such that $t_i = 0$ for each $i \in [4]$ and $q_j = m$ for each $j \in [3]$.  Lemma~\ref{lem: upper} then implies that $P^*_{DP}(G,m) = P_{DP}(G, \mathcal{H}) = m^4 - 6m^3 + 15m^2 - 13m.$ 

Now suppose that $m \geq 7$ and $m$ is odd.  To complete the proof, we must now show that the upper bound in Lemma~\ref{lem: upperodd} is attainable.  We will construct a full $m$-fold cover, $\mathcal{H} = (L,H)$, of $G = K_4$.  For each $i \in [4]$, let $L(v_i) = \{(v_i,j) : j \in [m] \}$ and let $V(H) = \bigcup_{i=1}^4 L(v_i)$.  Finally, using the notation above, create edges as follows.  First, create edges so that $\sigma_{1,2}$, $\sigma_{1,3}$, and $\sigma_{1,4}$ are the identity permutation on $[m]$.  Let $f$ be the permutation on $[m]$ given by $f(1)=2$, $f(2)=3$, $f(3)=1$, $f(x) = x+1$ whenever $x$ is even and $4 \leq x \leq m-1$, and $f(x) = x-1$ whenever $x$ is odd and $5 \leq x \leq m$.  Finally, create edges so that $\sigma_{2,3} = \sigma_{2,4} = f$ and $\sigma_{3,4} = f^{-1}$.  

Now, using the notation established above, it is easy to check that $\mathcal{H}$ is a full $m$-fold cover of $G$ such that $t_i = 0$ for each $i \in [4]$.  Also, using the notation above, if we suppose: the vertices of $Q_1$ in cyclic order are $v_1, v_2, v_3, v_4$, the vertices of $Q_2$ in cyclic order are $v_1, v_2, v_4, v_3$, and the vertices of $Q_3$ in cyclic order are $v_1, v_3, v_2, v_4$, then it is clear that $q_1 = q_3 = m$.  It is also easy to check that $q_2 = m-3$.   Finally, Identity~[\ref{identity}] and the fact that $\sum_{i=1}^4 t_i = 0$ implies $(3-m)\sum_{i=1}^4 t_i - \sum_{i=1}^6 m_i + z = 0$   which yields  $P_{DP}(G, \mathcal{H}) = m^4 - 6m^3 + 15m^2 - 13m - 3.$ 
\end{proof}

\section{Proof of Theorem~\ref{thm: generalize}} \label{more}

This section is organized as follows.  We begin by generalizing the notation and approach of Section~\ref{main} so that it applies to complete graphs with order larger than 4.  Then, we establish a lower and upper bound on $P_{DP}(G, \mathcal{H})$ when $G = K_n$ with $n \geq 4$ and $\mathcal{H}$ is an $m$-fold cover of $G$ with $m > 2^{t+1} + t + n - 6$ where $t = \binom{n}{2}$.  These bounds allow us to prove the bounds in Theorem~\ref{thm: generalize}.  Finally, we use the bounds in Theorem~\ref{thm: generalize} to show that when $\mathcal{H} = (L,H)$ is a full $m$-fold cover of $G$ satisfying $P_{DP}(G, \mathcal{H}) = P_{DP}^*(G,m)$ and $m > 2^{t+1} + t + n - 6$, $H$ is triangle-free.

We begin by generalizing the notation used in Section~\ref{main}.  This notation will be used for the remainder of the paper.  Suppose that $G = K_n$ with $n \geq 4$, $V(G) = \{v_1, \ldots, v_n \}$, $E(G) = \{e_1, \ldots, e_t \}$ where $t = \binom{n}{2}$.  Also, for some $m > 2^{t+1} + t + n - 6$ we always have that $\mathcal{H}= (L,H)$ is a full $m$-fold cover of $G$ such that $L(v) = \{(v,i): i \in [m] \}$ for each $v \in V(G)$.  Additionally, for $i, j \in [n]$ with $i < j$, we let $\sigma_{i,j}$ be the permutation of $[m]$ that maps each $q \in [m]$ to the second coordinate of the vertex in $L(v_j)$ that is adjacent to $(v_i,q)$ in $H$.  

Let $\mathcal{U} = \{ I \subseteq V(H) : |L(v) \cap I| = 1 \text{ for each } v \in V(G) \}$.  Clearly, $|\mathcal{U}| = m^n$.  Now, for $i \in [t]$, if $e_i=v_r v_s$, let $S_i$ be the set consisting of each $I \in \mathcal{U}$ with the property that $H[I]$ contains an edge in $E_H(L(v_r), L(v_s))$.   Also, for each $i \in [t]$ let $C_i = \mathcal{U} - S_i$.  Clearly,
$$P_{DP}(G, \mathcal{H}) = \left | \bigcap_{i=1}^t C_i \right |.$$
So, by the Inclusion-Exclusion Principle, we see that
$$P_{DP}(G, \mathcal{H}) = |\mathcal{U}| - \left | \bigcup_{i=1}^t S_i \right | =  m^n - \sum_{k=1}^t (-1)^{k-1} \left ( \sum_{1 \leq i_1 < \cdots < i_k \leq t} \left | \bigcap_{j=1}^k S_{i_j} \right| \right).$$ 
 
Now, suppose that $T_1, \ldots, T_a$ are the 3-cycles contained in $G$ (note that $a = \binom{n}{3}$).  For each $i \in [a]$, let $t_i$ be the number of $3$-cycles contained in the second coordinate of the subcover of $\mathcal{H}$ corresponding to $T_i$.  Next, suppose $Q_1, \ldots, Q_b$ are the 4-cycles contained in $G$ (note that $b = 3\binom{n}{4}$).  For each $i \in [b]$, let $q_i$ be the number of $4$-cycles contained in the second coordinate of the subcover of $\mathcal{H}$ corresponding to $Q_i$.  Suppose that $M_1, \ldots, M_{2b}$ are the copies of $K_4$ minus an edge contained in $G$.  For each $i \in [2b]$, let $m_i$ be the number of copies of $M_i$ contained in the second coordinate of the subcover of $\mathcal{H}$ corresponding to $M_i$.  Finally, suppose that $Z_1, \ldots, Z_{b/3}$ are the copies of $K_4$ contained in $G$.  For each $i \in [b/3]$, let $z_i$ be the number of copies of $Z_i$ contained in the second coordinate of the subcover of $\mathcal{H}$ corresponding to $Z_i$.  

We know from Section~\ref{main} that if for some $i \in [2b]$, $E(M_i)$ contains all the edges of $T_j$ and $T_r$ for some $j,r \in [a]$ with $j \neq r$, then $m_i \leq \min\{t_j , t_r\}$.  Also, if for some $i \in [b/3]$, $E(Z_i)$ contains all the edges of $T_j$, $T_q$, $T_r$, $T_s$ for some $j,q,r,s \in [a]$ with $j < q < r < s$, then $z_i \leq \min\{t_j , t_q, t_r, t_s\}$.  We are now ready to present some bounds and identities that will be essential in our proof of Theorem~\ref{thm: generalize}.

\begin{lem} \label{lem: formulas3}
Assuming the set-up established above, the following statements hold. \\
(i) For any $k \in [2]$ and $i_1, \ldots, i_k \in [t]$ satisfying $i_1 < \cdots < i_k$, $\left | \bigcap_{j=1}^k S_{i_j} \right|  = m^{n-k}$. \\  
(ii) If $e_{i_1}, \ldots, e_{i_3}$ are distinct edges that are not the edges of a $3$-cycle in $G$, $\left | \bigcap_{j=1}^3 S_{i_j} \right| = m^{n-3}$. \\
(iii) Suppose that $3 \leq k \leq 6$, and $i_1, \ldots, i_k \in [t]$ satisfy $i_1 < \cdots < i_k$.  Let $G'$ be the subgraph of $G$ with edge set $\{e_{i_1}, \ldots, e_{i_k} \}$ and vertex set consisting of only the vertices that are an endpoint of at least one element in $E(G')$.  Suppose $\mathcal{H}' = (L',H')$ is the subcover of $\mathcal{H}$ corresponding to $G'$.  Let $C(G')$ be the number of copies of $G'$ contained in $H'$.  If $G'$ is a 3-cycle in $G$, then $\left | \bigcap_{j=1}^3 S_{i_j} \right| = m^{n-3}C(G')$.  Furthermore, if $4 \leq k \leq 6$ and $|V(G')|=4$, then $\left | \bigcap_{j=1}^k S_{i_j} \right| = m^{n-4}C(G').$\\   
(iv) Suppose $G'$ is the spanning subgraph of $G$ such that $E(G')$ contains all the edges of $T_j$ for some $j \in [a]$ and one additional edge.  Suppose  $E(G') = \{e_{i_1}, \ldots, e_{i_4} \}$.  Then, $\left | \bigcap_{j=1}^4 S_{i_j} \right|  = t_j m^{n-4}$. \\
(v) Suppose that $4 \leq k \leq 6$, and $i_1, \ldots, i_k \in [6]$ satisfy $i_1 < \cdots < i_k$.  Let $G'$ be the subgraph of $G$ with edge set $\{e_{i_1}, \ldots, e_{i_k} \}$ and vertex set consisting of only the vertices that are an endpoint of at least one element in $E(G')$.  If $G'$ is not a 3-cycle plus an edge, $G' \neq Q_i$ for some $i \in [b]$, $G' \neq M_i$ for some $i \in [2b]$, and $G' \neq Z_i$ for some $i \in [b/3]$, then $\left | \bigcap_{j=1}^k S_{i_j} \right| \leq m^{n-4}.$\\ 
(vi) For $k \geq 7$, $\sum_{1 \leq i_1 < \cdots < i_k \leq t} \left | \bigcap_{j=1}^k S_{i_j} \right| \leq \binom{t}{k} m^{n-4}$.  
\end{lem}

\begin{proof}
Statements (i), (ii), and (vi) are proven in a more general setting in~\cite{MT20}.  So, we begin with a proof of Statement~(iii).  Let $\mathcal{C}(G')$ be the set consisting of the vertex set of each copy of $G'$ in $H'$.  We suppose $G'$ is a 3-cycle or $4 \leq k \leq 6$ and $|V(G')|=4$.  Note that $\bigcap_{j=1}^k S_{i_j}$ has an element if and only if $\mathcal{C}(G')$ has an element. So, both results are clear when $C(G') = 0$, and we suppose that $C(G') > 0$. 

Suppose $C = V(G) - V(G')$, and $c = |C|$ (note that $c = n-3$ when $G'$ is a 3-cycle, and $c = n-4$ when $|V(G')|=4$).  Name the elements of $C$ so that $C = \{c_1, \ldots, c_c\}$.  Let $M : \mathcal{C}(G') \times [m]^c \rightarrow \bigcap_{j=1}^k S_{i_j}$ be the function given by $M(V,(x_1, \ldots, x_c)) = V \cup \{(c_i,x_i) : i \in [c] \}$.  It is easy to verify that $M$ is a bijection which yields $\left | \bigcap_{j=1}^k S_{i_j} \right| = m^{c}C(G')$.

We now turn our attention to Statement~(iv).  Suppose without loss of generality that $e_{i_4} \notin E(T_j)$.  We may suppose that $e_{i_4}$ doesn't share an endpoint with any element of $E(T_j)$; otherwise, the desired result follows from Lemma~\ref{lem: plusedge} and Statement~(iii).  Suppose $e_{i_4} = xy$.  Let $M : \bigcap_{j=1}^4 S_{i_j} \times [m] \rightarrow \bigcap_{j=1}^3 S_{i_j}$ be the function given by $M(V,i) = (V - (V \cap L(y))) \cup \{(y,i)\}$.  One can verify that $M$ is a bijection which yields $m\left | \bigcap_{j=1}^4 S_{i_j} \right| = \left | \bigcap_{j=1}^3 S_{i_j} \right|$.  Since Statement~(iii) implies $\left| \bigcap_{j=1}^3 S_{i_j} \right| = t_j m^{n-3}$, our proof of Statement~(iv) is complete.

For Statement~(v) notice that for any $G'$ satisfying the hypotheses, the spanning subgraph of $G$ with edge set $E(G')$ has at least $n-4$ components.  The result immediately follows (this idea is also used in~\cite{MT20} to prove Statement~(vi)).
\end{proof}

Lemma~\ref{lem: formulas3} along with the fact there are $(t-3)$ spanning subgraphs of $G$ that contain all the edges of $T_j$ and one additional edge, now yields the following.

\begin{align*}
&P_{DP}(G, \mathcal{H}) \\
&=  m^n + \sum_{k=1}^t (-1)^{k} \left ( \sum_{1 \leq i_1 < \cdots < i_k \leq t} \left | \bigcap_{j=1}^k S_{i_j} \right| \right) \\
&= m^n - tm^{n-1} + \binom{t}{2} m^{n-2} - \left(\binom{t}{3}-\binom{n}{3} \right)m^{n-3} - \left(\sum_{i=1}^a t_i \right)m^{n-3} \\
&+ \sum_{k=4}^6 (-1)^{k} \left ( \sum_{1 \leq i_1 < \cdots < i_k \leq t} \left | \bigcap_{j=1}^k S_{i_j} \right| \right) + \sum_{k=7}^t (-1)^{k} \left ( \sum_{1 \leq i_1 < \cdots < i_k \leq t} \left | \bigcap_{j=1}^k S_{i_j} \right| \right) \\
&\geq m^n - tm^{n-1} + \binom{t}{2} m^{n-2} - \left(\binom{t}{3}-\binom{n}{3} \right)m^{n-3} - \left(\sum_{i=1}^a t_i \right)m^{n-3} + \left(\sum_{i=1}^b q_i \right)m^{n-4}\\
&+ (t-3)\left(\sum_{i=1}^a t_i \right)m^{n-4} - \left(\sum_{i=1}^{2b} m_i \right)m^{n-4} + \left(\sum_{i=1}^{b/3} z_i \right)m^{n-4}  \\
&- \left(\binom{t}{4}-b - a(t-3) \right)m^{n-4} - \left(\binom{t}{5}- 2b \right)m^{n-4} - \left(\binom{t}{6}-b/3 \right)m^{n-4}  - \sum_{k=7}^t \binom{t}{k} m^{n-4} \\
&\geq m^n - tm^{n-1} + \binom{t}{2} m^{n-2} - \left(\binom{t}{3}-\binom{n}{3} \right)m^{n-3} + (t-3-m) \left(\sum_{i=1}^a t_i \right)m^{n-4}  \\
&+ \left(\sum_{i=1}^b q_i \right)m^{n-4} - \left(\sum_{i=1}^{2b} m_i \right)m^{n-4} + \left(\sum_{i=1}^{b/3} z_i \right)m^{n-4} \\ 
&+ (b + a(t-3) + 2b + b/3)m^{n-4}- 2^{t}m^{n-4}.  \numberthis \label{lower}
\end{align*}

Similarly, using the fact that $m > 2^{t+1} + t + n - 6$ which means $n+t-m - 6< 0$, $q_i \leq m$ for each $i \in [b]$, and for each $i \in [a]$, $T_i$ is a subgraph of exactly $(n-3)$ of the graphs: $Z_1, \ldots, Z_{b/3}$, we obtain:  

\begin{align*}
&P_{DP}(G, \mathcal{H}) \\
&= m^n - tm^{n-1} + \binom{t}{2} m^{n-2} - \left(\binom{t}{3}-\binom{n}{3} \right)m^{n-3} - \left(\sum_{i=1}^a t_i \right)m^{n-3} \\
&+ \sum_{k=4}^6 (-1)^{k} \left ( \sum_{1 \leq i_1 < \cdots < i_k \leq t} \left | \bigcap_{j=1}^k S_{i_j} \right| \right) + \sum_{k=7}^t (-1)^{k} \left ( \sum_{1 \leq i_1 < \cdots < i_k \leq t} \left | \bigcap_{j=1}^k S_{i_j} \right| \right) \\
&\leq m^n - tm^{n-1} + \binom{t}{2} m^{n-2} - \left(\binom{t}{3}-\binom{n}{3} \right)m^{n-3} - \left(\sum_{i=1}^a t_i \right)m^{n-3} + \left(\sum_{i=1}^b q_i \right)m^{n-4}\\
&+ (t-3)\left(\sum_{i=1}^a t_i \right)m^{n-4} - \left(\sum_{i=1}^{2b} m_i \right)m^{n-4} + \left(\sum_{i=1}^{b/3} z_i \right)m^{n-4} + \sum_{k=4}^t \binom{t}{k} m^{n-4} \\
&\leq m^n - tm^{n-1} + \binom{t}{2} m^{n-2} - \left(\binom{t}{3}-\binom{n}{3} \right)m^{n-3} + (t-3-m) \left(\sum_{i=1}^a t_i \right)m^{n-4} \\
&+ bm^{n-3} + (n-3)\left(\sum_{i=1}^{a} t_i \right)m^{n-4} + 2^{t}m^{n-4} \numberthis \label{upper}\\
&\leq m^n - tm^{n-1} + \binom{t}{2} m^{n-2} - \left(\binom{t}{3}-\binom{n}{3} - 3 \binom{n}{4} \right)m^{n-3} +  2^{t}m^{n-4} .
\end{align*}

Notice that this establishes the upper bound in Theorem~\ref{thm: generalize}.  We will now construct covers and use bound~(\ref{lower}) to establish the lower bound in Theorem~\ref{thm: generalize}.

\begin{lem} \label{lem: lower}
Suppose $G = K_n$ with $n \geq 4$, and let $t = \binom{n}{2}$.  When $m > 2^{t+1} + t + n - 6$,
$$m^n - t m^{n-1} + \binom{t}{2} m^{n-2} - \left(\binom{t}{3} - \binom{n}{3} - 3\binom{n}{4} \right) m^{n-3} - 2^{t}m^{n-4} \leq P^*_{DP}(G,m).$$
\end{lem}

\begin{proof}
We will first prove the result when $m$ is even.  Construct a full $m$-fold cover, $\mathcal{H} = (L,H)$, of $G$ as follows.  For each $i \in [n]$, let $L(v_i) = \{(v_i,j) : j \in [m] \}$ and let $V(H) = \bigcup_{i=1}^n L(v_i)$.  Finally, create edges in $H$ so that whenever $i, j \in [t]$ and $i < j$, $(v_i,t)(v_j,t+1) \in E(H)$ when $t$ is odd and $(v_i,t)(v_j,t-1) \in E(H)$ when $t$ is even.  Now, using the notation established above, it is easy to check that $\mathcal{H}$ is a full $m$-fold cover of $G$ such that $t_i = 0$ for each $i \in [a]$ and $q_j = m$ for each $j \in [b]$.  The fact that $t_i = 0$ for each $i \in [a]$ also implies that $m_i = 0$ for each $i \in [2b]$ and $z_i = 0$ for each $i \in [b/3]$.  Our desired bound now immediately follows from bound~(\ref{lower}) and the fact that $P_{DP}(G,\mathcal{H}) \leq P_{DP}^*(G,m)$.

Now suppose that $m$ is odd. Construct a full $m$-fold cover, $\mathcal{H} = (L,H)$, of $G$ as follows.  For each $i \in [n]$, let $L(v_i) = \{(v_i,j) : j \in [m] \}$ and let $V(H) = \bigcup_{i=1}^n L(v_i)$.  Let $f$ be the permutation on $[m]$ given by $f(1)=2$, $f(2)=3$, $f(3)=1$, $f(x) = x+1$ whenever $x$ is even and $4 \leq x \leq m-1$, and $f(x) = x-1$ whenever $x$ is odd and $5 \leq x \leq m$.  Now, create edges so that whenever $i, j \in [t]$ and $i < j$, $\sigma_{i,j} = f$.  

Now, using the notation established above, we have that $t_i = 0$ for each $i \in [a]$ since $f^{-1}(f(f(x))) = f(x) \neq x$ for each $x \in [m]$.  As in the case when $m$ is even, this implies $m_i = 0$ for each $i \in [2b]$ and $z_i = 0$ for each $i \in [b/3]$.  

Now, suppose $v_i, v_j, v_k, v_l$ are four arbitrary vertices in $G$ such that $i < j < k < l$.  Additionally, suppose $Q_q$, $Q_r$, and $Q_s$ are the 4-cycles in $G$ with vertex set equal to $\{v_i, v_j, v_k, v_l\}$.  We will now compute $q_q$, $q_r$, and $q_s$.  Without loss of generality we may assume $i=1$, $j=2$, $k=3$, and $l=4$.  We may also suppose the vertices of $Q_q$ in cyclic order are $v_1, v_2, v_3, v_4$, the vertices of $Q_r$ in cyclic order are $v_1, v_2, v_4, v_3$, and the vertices of $Q_s$ in cyclic order are $v_1, v_3, v_2, v_4$.  Note that for each $x \in [m]$, 
$$ \sigma^{-1}_{1,4}(\sigma_{3,4}(\sigma_{2,3}(\sigma_{1,2}(x)))) = f^{-1}(f(f(f(x)))) = f(f(x)).$$
Since $f(f(x))=x$ for each $x \in [m]-[3]$ and $f(f(x)) \neq x$ for each $x \in [3]$, we have that $q_q = m-3$.  Similarly, since $f^{-1}(f^{-1}(f(f(x))))=x$ for each $x \in [m]$, $q_r = m$; additionally, since $f^{-1}(f(f^{-1}(f(x))))=x$ for each $x \in [m]$, $q_s = m$.  This means
$$\left(\sum_{i=1}^b q_i \right) m^{n-4} = (mb-b) m^{n-4} = bm^{n-3} - b m^{n-4}.$$
Our desired bound now immediately follows from bound~(\ref{lower}) and the fact that $P_{DP}(G,\mathcal{H}) \leq P_{DP}^*(G,m)$.
\end{proof}

We are now ready to complete the proof of Theorem~\ref{thm: generalize} which we restate.

\begin{customthm} {\ref{thm: generalize}}
Suppose $G = K_n$ with $n \geq 4$, and let $t = \binom{n}{2}$.  If $m > 2^{t+1} + t + n - 6$ and $\mathcal{H} = (L,H)$ is a full $m$-fold cover of $G$ satisfying $P_{DP}(G, \mathcal{H}) = P_{DP}^*(G,m)$, then $H$ is triangle-free.  Moreover, when $m > 2^{t+1} + t + n - 6$,
$$f(m) - 2^{t}m^{n-4} \leq P^*_{DP}(G,m) \leq  f(m) + 2^{t}m^{n-4}$$
where 
$$f(m) = m^n - t m^{n-1} + \binom{t}{2} m^{n-2} - \left(\binom{t}{3} - \binom{n}{3} - 3\binom{n}{4} \right) m^{n-3}.$$
\end{customthm}

\begin{proof}
Suppose $\mathcal{H} = (L,H)$ is a full $m$-fold cover of $G$ satisfying $P_{DP}(G, \mathcal{H}) = P_{DP}^*(G,m)$ and $m > 2^{t+1} + t + n - 6$.  Since we have already proven both bounds on $P_{DP}^*(G,m)$ above, we need only show that $H$ is triangle-free.

For the sake of contradiction, suppose that $H$ is not triangle-free.  Using the notation above, this implies that $\sum_{i=1}^a t_i > 0$.  Let $S = \sum_{i=1}^a t_i$.  Then, by Bound~(\ref{upper}) and the fact that $m > 2^{t+1} + t + n - 6$, we obtain

\begin{align*}
P_{DP}^*(G,m) &= P_{DP}(G, \mathcal{H}) \\
&\leq m^n - tm^{n-1} + \binom{t}{2} m^{n-2} - \left(\binom{t}{3}-\binom{n}{3} - 3 \binom{n}{4} \right)m^{n-3} \\
&+(t+n-6 - m)Sm^{n-4} + 2^{t}m^{n-4} \\
&\leq m^n - tm^{n-1} + \binom{t}{2} m^{n-2} - \left(\binom{t}{3}-\binom{n}{3} - 3 \binom{n}{4} \right)m^{n-3} \\
&+(2^t+ t+n-6 - m)m^{n-4} \\
&< m^n - tm^{n-1} + \binom{t}{2} m^{n-2} - \left(\binom{t}{3}-\binom{n}{3} - 3 \binom{n}{4} \right)m^{n-3} - 2^{t}m^{n-4}.
\end{align*}

This however contradicts $f(m) - 2^t m^{n-4} \leq P_{DP}^*(G,m)$.
\end{proof}

{\bf Acknowledgment.}  The authors would like to thank Hemanshu Kaul for his guidance and encouragement.

\section{Appendix} \label{code}

The program in this Appendix is written in Python, and it is available at~\cite{S24}.  The program can be used to compute $P_{DP}^*(K_4,m)$ when $m=2,3,4,5$.  For our inputs for the program we assume $G = K_4$, $V(G) = \{v_1, v_2, v_3, v_4 \}$, and $E(G) = \{e_1, e_2, e_3, e_4, e_5, e_6 \}$ where $e_1 = v_1v_2$, $e_2 = v_1v_3$, $e_3 = v_1v_4$, $e_4 = v_2v_3$, $e_5 = v_2v_4$, and $e_6 = v_3v_4$.  By Corollary~12 in~\cite{DKM23}, $P_{DP}^*(G,m)$ can be computed by considering all full $m$-fold covers for $G$ such that $\sigma_{1,2}$, $\sigma_{1,3}$, and $\sigma_{1,4}$ are all the identity permutation. 

\small{
\begin{verbatim}
#################
#
#   GUIDE
# 1.) Ask $|V(G)|$ and $|E(G)|$ from the user.
# 2.) Ask the user for the incidence matrix of $G$.
# 3.) Asks the user for the value of $m$.
# 4.) For each edge, $e = v_i v_j$ with $i < j$ the user indicates with a $0$ 
#     that $\sigma_{i,j}$ will be fixed as the identity permutation for all 
#     full $m$-fold covers that the program considers, or the user 
#     indicates with a $1$ that $\sigma_{i,j}$ is not fixed. 
#
#   PROCESS
#     The program counts the number of proper $\mathcal{H}$-colorings of $G$ 
#     for each full $m$-fold cover, $\mathcal{H}$ of $G$, satisfying: 
#     for any edge $v_i v_j$ with $i < j$ that received a 0 in the 4th step, 
#     $\mathcal{H}$ has $\sigma_{i,j}$ as the identity permutation.
#     Thus if $q$ is the number of 1's entered by the user,
#     the program considers $(m!)^q$ full $m$-fold covers of $G$.
#
#   OUTPUT
#     The first output will be the maximum number of proper colorings 
#     over all $m$-fold covers that the program considers.
#
#     The second output will be the minimum number of proper colorings 
#     over all $m$-fold covers that the program considers.
# 
# To verify Theorem 3 for $m = 2,3,4,5$ use the following inputs:
# 1.) Enter the number of vertices: 4
#     Enter the number of edges: 6
#
# 2.) Enter the incidence matrix(one row at a time and space-separated values)
#           E1 E2 E3 E4 E5 E6
#       V1: 1 1 1 0 0 0
#       V2: 1 0 0 1 1 0
#       V3: 0 1 0 1 0 1
#       V4: 0 0 1 0 1 1
#
# 3.) Enter the fold number: m
#
# 4.) Enter the edges that will correspond to the identity permutation in the cover 
#     ('0' for identity, '1' otherwise):
#       E1: 0
#       E2: 0
#       E3: 0
#       E4: 1
#       E5: 1
#       E6: 1
#
#   OUTPUTS FOR EACH VALUE OF $m$ 
#       When $m = 2$:
#       2
#       0
#
#       When $m = 3$:
#       12
#       0
#
#       When $m = 4$:
#       60
#       24
#
#       When $m = 5$:
#       182
#       120
#
#################

import math as math
import itertools as iter
from tqdm import tqdm

#
# Creation of the incidence matrix of $G$. Returns a 2D Matrix.
# When $G$ has $l$ vertices and $w$ edges, the user specifies $l$ and $w$
# and then the incidence matrix.
# 
# 
def incidenceM_Creation(l, w):
    # Initialize the incidence matrix with zeros
    matrix = [[0] * w] * l

    print("Enter the incidence matrix "
          "(one row at a time and space-separated values):")
    for cols in range(w):
        if (cols == 0):
            print(f"    E{cols + 1}", end=" ")
        else:
            print(f"E{cols + 1}", end=" ")
    for i in range(l):
        print()
        row = input(f"V{i + 1}: ").strip().split()
        if len(row) != w:
            print("Error: Number of elements in the row does "
                  "not match the number of columns.\n")
            incidenceM_Creation(l, w)
            return None
        matrix[i] = [int(val) for val in row]
    print()   
    return matrix

# 
# Returns a 1D Matrix
# 
# For each edge, $e = v_i v_j$ with $i < j$ the user indicates with a $0$ 
#     that $\sigma_{i,j}$ will be fixed as the identity permutation for all 
#     full $m$-fold covers that the program considers, or the user 
#     indicates with a $1$ that $\sigma_{i,j}$ is not fixed.
# 
def edgeID_Creation(w):
    matrix = []
    print("\nEnter the edges that will correspond to the identity permutation "
          "in the cover('0' for identity, '1' otherwise): ")
    for i in range(w):
        matrix.append(int(input(f"Edge {i + 1}: ")))
        print()
    return matrix

#
# Generates every possible coloring based on the second coordinates of 
# the elements of $L(v)$ for each $v \in V(G)$.
#
def generate_colorings(n, length):
    return [list(perm) for perm in iter.product(range(1, n + 1), 
                                                repeat=length)]

#
# Determines whether a coloring is proper.
#
def coloring_function(m, v, e, incidence_matrix, p_edges, colorings):
    c = 0
    for i in range(m**v):
        z = 0
        for j in range(e):
            a = -1
            b = -1
            for k in range(v):
                if (incidence_matrix[k][j] == 1):
                    a = b
                    b = k
            if (p_edges[j][colorings[i][a] - 1] != colorings[i][b]):
                z += 1
        if (z == e):
            c += 1
    return c

##### DRIVER PROGRAM ######

# Ask $|V(G)|$ and $|E(G)|$ from the user.
v = int(input("Enter the number of vertices: "))
e = int(input("Enter the number of edges: "))

# Ask the user for the incidence matrix of $G$.
incidence_matrix = incidenceM_Creation(v, e)

# Ask the user for the value of $m$.
m = int(input("\nEnter the fold number: "))


# Generates every possible coloring based on the second coordinates of 
# the elements of $L(v)$ for each $v \in V(G)$.
colorings = generate_colorings(m, v)

# D is the maximum number of proper colorings 
# over all $m$-fold covers that the program considers.

# d is the minimum number of proper colorings 
# over all $m$-fold covers that the program considers.
D = 0
d = (m**v)

# Generates all permutations of $[m]$.
permutations = list(iter.permutations(range(1, m + 1)))


# Enter the edges that will correspond to the identity permutation 
# in the cover ('0' for identity, '1' otherwise):
edges_id = edgeID_Creation(e)

# Determine the number of 1's entered by the user.
q = 0 
for i in range(e):
    q += edges_id[i]

# Generates all full $m$-fold covers that the program will consider.
covers = generate_colorings(math.factorial(m), q)

#
# Counts the number of proper $\mathcal{H}$-colorings of $G$ 
# for each full $m$-fold cover, $\mathcal{H}$ of $G$, satisfying: 
# for any edge $v_i v_j$ with $i < j$ that received a 0, 
# $\mathcal{H}$ has $\sigma_{i,j}$ as the identity permutation.
# Thus if $q$ is the number of 1's entered by the user,
# the program considers $(m!)^q$ full $m$-fold covers of $G$. 
#
p = [[0] * m] * e
for i in tqdm (range(math.factorial(m)**q), desc="Calculating..."):
    for j in range(e):
        count = 0
        if edges_id[j]== 1:
            p[j] = permutations[covers[i][count] - 1]
            count += 1
        else:
            p[j] = permutations[0]
    c = coloring_function(m, v, e, incidence_matrix, p, colorings)
    if c > D:
        D = c
    if c < d:
        d = c

# OUTPUTS
print(f'Max: {D}')
print(f'Min: {d}')
\end{verbatim}}

\end{document}